\documentclass[11pt]{amsart}
\usepackage[a4paper,margin=1in]{geometry}
\usepackage{amsmath,amsthm,amssymb,mathtools}
\usepackage[T1]{fontenc}
\usepackage{lmodern}
\usepackage[
  colorlinks=true,
  linkcolor=blue,
  citecolor=purple,
  urlcolor=blue,
  pagebackref=true
]{hyperref}
\usepackage{enumitem}
\usepackage{xcolor}
\newtheorem{theorem}{Theorem}[section]
\newtheorem{proposition}[theorem]{Proposition}
\newtheorem{lemma}[theorem]{Lemma}
\newtheorem{corollary}[theorem]{Corollary}
\theoremstyle{remark}
\newtheorem{remark}[theorem]{Remark}

\newcommand{\SL}{\mathrm{SL}}
\newcommand{\PSL}{\mathrm{PSL}}

\newcommand{\ZZ}{\mathbb Z}
\newcommand{\QQ}{\mathbb Q}

\newcommand{\eps}{\varepsilon}

\newcommand{\comm}[2]{[#1,#2]}
\DeclareMathOperator{\ord}{ord}
\DeclareMathOperator{\tr}{tr}
\DeclareMathOperator{\Cay}{Cay}

\title[Explicit uniform two-generator presentations]%
{Uniform two-generator presentations for $\SL_n(\ZZ)$
with polynomial complexity bounds}

\author{Arindam Biswas}

\address{}
\email{arin.math@gmail.com}

\date{\today}



\begin{document}

\begin{abstract}
We give a uniform explicit construction of finite two-generator presentations
for the special linear groups over the integers in all ranks at least three.
The construction builds on the generating-pair work of
Conder--Liversidge--Vsemirnov and on a standard Tietze-elimination observation pointed out by Button. It recovers Trott's odd-rank generating pair and extends
the same monomial/transvection form uniformly to even rank by a sign correction.
After rebalancing, the construction has quadratic transvection words, quartically
many relators, and sextic total relator length. We also derive several consequences, including infinite--infinite and finite--finite
variants, consequences for congruence quotients, a presentation for the projective
quotient, and an exact relator count, valid for both the unbalanced and balanced
presentations.
\end{abstract}

\maketitle

\section{Introduction}

For \(n\ge 3\), let \(E_{i,j}\) be the standard matrix unit and let
\[
T_{i,j}=I_n+E_{i,j}\in \SL_n(\ZZ)\qquad (i\neq j)
\]
be the elementary transvection.  Throughout, we use the commutator convention
\[
\comm{x}{y}=x^{-1}y^{-1}xy,
\]
and an equation \(U=V\) in a presentation is interpreted as the relator
\(UV^{-1}\).  We shall use the standard Steinberg presentation
\begin{equation}\label{eq:steinberg}
\begin{split}
\SL_n(\ZZ)\cong \big\langle T_{i,j}\ (i\neq j)\ \big|\ &
[T_{i,j},T_{k,l}]=1\ \text{if }i\neq l,\ j\neq k,\\
&
[T_{i,j},T_{j,k}]=T_{i,k}\ \text{if }i,j,k\text{ distinct},\\
&
(T_{1,2}T_{2,1}^{-1}T_{1,2})^4=1
\big\rangle .
\end{split}
\end{equation}
See, for example, \cite{Milnor,Conder}.

Several explicit generating pairs for \(\SL_n(\ZZ)\) are known.  For odd \(n\),
Trott constructed a generating pair consisting of a cyclic permutation matrix and
a transvection \cite{Trott}.  Gow and Tamburini proved that \(\SL_n(\ZZ)\) is
generated by a Jordan unipotent block and its transpose for all \(n>2\) with
\(n\neq 4\) \cite{GowTamburini}.  More recently,
Conder--Liversidge--Vsemirnov exhibited several explicit generating pairs for
\(\SL_n(\ZZ)\), and in rank \(3\) gave eight explicit finite two-generator
presentations for \(\SL_3(\ZZ)\) \cite{Conder}. The formal route from a known presentation to a presentation on a new generating
set is the following.  If the old generators are given as words in the new generators,
and the new generators are also given as words in the old ones, then Tietze
transformations give a presentation on the new generators by substitution,
together with the inverse relations.  This observation was pointed out by Button
in his MathSciNet review of Conder--Liversidge--Vsemirnov \cite{ButtonMR}, where
it is credited to Philip Hall.  In rank \(3\), the finite two-generator
presentations of Conder--Liversidge--Vsemirnov arise from this substitution
procedure, applied to explicit generating pairs for \(\SL_3(\ZZ)\)
\cite{Conder}. 

One of the motivation of the present paper is to carry out this procedure uniformly
for all ranks \(n\ge 3\) for one explicit Trott-type family, and to do so with
polynomial control on the resulting words and relators.  To the best of our
knowledge, no single natural all-rank generating family for \(\SL_n(\ZZ)\),
\(n\ge 3\), had previously been accompanied by explicit all-rank transvection
words with uniform polynomial length bounds, and hence by a uniform family of
two-generator presentations with polynomially bounded number of relators and
total relator length.

We provide such an implementation for the following Trott-type family.  Define
\[
\eps_n:=(-1)^{n-1},
\]
and set
\begin{equation}\label{eq:anbn}
a_n:=\sum_{i=1}^{n-1}E_{i,i+1}+\eps_n E_{n,1},
\qquad
b_n:=I_n+E_{2,1}.
\end{equation}
Thus
\[
a_n=
\begin{pmatrix}
0&1&0&\cdots&0\\
0&0&1&\ddots&\vdots\\
\vdots&\ddots&\ddots&\ddots&0\\
0&\cdots&0&0&1\\
\eps_n&0&\cdots&0&0
\end{pmatrix},
\qquad
b_n=
\begin{pmatrix}
1&0&0&\cdots&0\\
1&1&0&\cdots&0\\
0&0&1&\ddots&\vdots\\
\vdots&\vdots&\ddots&\ddots&0\\
0&0&\cdots&0&1
\end{pmatrix}.
\]
When \(n\) is odd, this is Trott's pair \cite{Trott}.  When \(n\) is even, the
sign correction \(\eps_n=-1\) gives the same monomial/transvection shape inside
\(\SL_n(\ZZ)\).  Hence \((a_n,b_n)\) gives a single all-rank family, and in
particular avoids the rank-\(4\) exception in the Gow--Tamburini
Jordan/transposed-Jordan construction \cite{GowTamburini}.

The main step is to write every elementary transvection \(T_{i,j}\) explicitly
as a word in \(a_n,b_n\).  Substituting these words into the Steinberg
presentation gives a finite two-generator presentation of \(\SL_n(\ZZ)\).  We
then rebalance the transvection words and prove the quantitative bounds
\[
\ell(T_{i,j})=O(n^2),
\]
in the generators \(a_n,b_n\).  Consequently the resulting presentations have
\(O(n^4)\) relators and total relator length \(O(n^6)\).  The same explicit
substitution also yields infinite--infinite and finite--finite variants,
generation of congruence quotients, a presentation for the projective quotient,
and an exact relator count.

The paper is organised as follows.  Section~\ref{sec:uniform} constructs the
words for all elementary transvections in the pair \((a_n,b_n)\) and derives the
corresponding two-generator presentation from the Steinberg presentation.
Section~\ref{sec:quantitative} replaces the initial recursive words by balanced
ones and proves the \(O(n^2)\), \(O(n^4)\), and \(O(n^6)\) bounds.
Section~\ref{sec:low-rank-variants} discusses low-rank remarks and gives
explicit infinite--infinite and finite--finite variants.  Section~\ref{sec:quotients}
derives the congruence quotient statement, the projective quotient presentation,
and the exact relator count.  The final section compares the construction with
earlier explicit generating families.

\section{A uniform construction from the Steinberg presentation}\label{sec:uniform}

\subsection{A Tietze-elimination template}

We begin by recalling the standard formal mechanism that will be used throughout.
This Tietze-elimination observation was pointed out by
Button in \cite{ButtonMR}, with attribution there to Philip Hall.

\begin{proposition}\label{prop:tietze-template}
Let $G$ be a group admitting a finite presentation
\[
G\cong \langle x_1,\dots,x_m\mid R_1,\dots,R_s\rangle.
\]
Here each relator $R_j$ is regarded as a chosen word in the free group
$F(x_1,\dots,x_m)$ representing the corresponding defining relation.
Assume that $a,b\in G$ generate $G$, and that for each $i$ there is an explicit word
$w_i(a,b)\in F(a,b)$ such that $x_i=w_i(a,b)$ in $G$. Assume moreover that
$a=u(x_1,\dots,x_m)$ and $b=v(x_1,\dots,x_m)$ in $G$ for explicit words
$u,v\in F(x_1,\dots,x_m)$. Then
\[
\begin{aligned}
G\cong \big\langle a,b\ \big|\ &
R_1(w_1(a,b),\dots,w_m(a,b)),\dots,R_s(w_1(a,b),\dots,w_m(a,b)),\\
&
a=u(w_1,\dots,w_m),\ b=v(w_1,\dots,w_m)
\big\rangle.
\end{aligned}
\]
\end{proposition}

\begin{proof}
Starting from the given presentation of $G$, first adjoin new generators $a,b$ with the
relations
\[
a=u(x_1,\dots,x_m),\qquad b=v(x_1,\dots,x_m).
\]
This is a Tietze expansion, so it does not change the presented group. In the resulting
presentation the relations
\[
x_i=w_i(a,b)\qquad (1\le i\le m)
\]
hold as identities in the presented group, because this expanded presentation still
presents $G$ and the words $x_iw_i(a,b)^{-1}$ are trivial there by hypothesis. Hence one
may adjoin these relations as consequences, and then eliminate the generators
$x_1,\dots,x_m$ using them. The resulting presentation is exactly the one displayed
above, and it presents the same group.
\end{proof}

\begin{remark}
Proposition~\ref{prop:tietze-template} is included to fix notation and to make
the later substitutions completely explicit. The proposition is the standard
Tietze-elimination step attributed above.
\end{remark}

We shall apply Proposition~\ref{prop:tietze-template} with the generators $x_i$ taken to
be the elementary transvections $T_{i,j}$ and with the defining relations coming from the
Steinberg presentation \eqref{eq:steinberg}.

\subsection{Elementary matrix identities}

For later use, we note the basic matrix identities underlying the Steinberg relations.
Recall that the matrix units satisfy
\[
E_{i,j}E_{k,l}=\delta_{j,k}E_{i,l},
\]
where 

$$\delta_{j,k} = \begin{cases}
1 & \text{if } j = k, \\
0 & \text{if } j \neq k.
\end{cases}$$

Since $T_{i,j}=I_n+E_{i,j}$, the commutator relations needed below can be checked
directly from this identity.

\begin{lemma}\label{lem:elementary-identities}
Let $i\neq j$ and $k\neq l$.
\begin{enumerate}[label=(\roman*)]
\item If $i\neq l$ and $j\neq k$, then
\[
T_{i,j}T_{k,l}=T_{k,l}T_{i,j}.
\]
\item If $i,j,k$ are distinct, then
\[
[T_{i,j},T_{j,k}]=T_{i,k}.
\]
\end{enumerate}
\end{lemma}

\begin{proof}
For part (i), one has
\[
T_{i,j}T_{k,l}
=(I_n+E_{i,j})(I_n+E_{k,l})
=I_n+E_{i,j}+E_{k,l}+E_{i,j}E_{k,l}.
\]
If $j\neq k$, then $E_{i,j}E_{k,l}=0$. Similarly,
\[
T_{k,l}T_{i,j}
=I_n+E_{k,l}+E_{i,j}+E_{k,l}E_{i,j},
\]
and if $l\neq i$, then $E_{k,l}E_{i,j}=0$. Under the stated hypotheses the two products
are therefore equal.

For part (ii), since $i,j,k$ are distinct, one has
\[
E_{i,j}^2=E_{j,k}^2=0,\qquad E_{j,k}E_{i,j}=0,\qquad E_{i,j}E_{j,k}=E_{i,k}.
\]
Hence
\[
T_{i,j}^{-1}=I_n-E_{i,j},\qquad T_{j,k}^{-1}=I_n-E_{j,k},
\]
and therefore
\[
[T_{i,j},T_{j,k}]
=(I_n-E_{i,j})(I_n-E_{j,k})(I_n+E_{i,j})(I_n+E_{j,k}).
\]
A direct multiplication gives
\[
[T_{i,j},T_{j,k}]=I_n+E_{i,k}=T_{i,k}.
\]
\end{proof}

\subsection{A uniform generating pair}

Although our main interest is $n\ge 4$, the same argument works for $n=3$, so we develop
it uniformly.

Fix an integer $n\ge 3$, and define $a_n,b_n$ by \eqref{eq:anbn}.  Write simply
$a=a_n$ and $b=b_n$ when the rank is clear.

\begin{lemma}\label{lem:det}
One has $a,b\in \SL_n(\ZZ)$.
\end{lemma}

\begin{proof}
The matrix $b$ is a transvection, hence has determinant $1$.
The matrix $a$ is monomial, corresponding to the $n$-cycle $(1\,2\,\dots\,n)$ with the
scalar $\eps_n$ on the wrap-around entry.  Therefore
\[
\det(a)=(-1)^{n-1}\eps_n=1.
\]
\end{proof}

\begin{proposition}\label{prop:orders}
For every $n\ge 3$, the generator $b_n$ has infinite order.
Moreover:
\begin{enumerate}[label=(\roman*)]
\item if $n$ is odd, then $a_n$ has order $n$;
\item if $n$ is even, then $a_n$ has order $2n$.
\end{enumerate}
In particular, when $n$ is odd the pair $(a_n,b_n)$ consists of elements of orders $n$
and $\infty$, while when $n$ is even it consists of elements of orders $2n$ and
$\infty$.
\end{proposition}

\begin{proof}
Since $b_n=I_n+E_{2,1}$ and $E_{2,1}^2=0$, one has
\[
b_n^m=(I_n+E_{2,1})^m=I_n+mE_{2,1}
\]
for every integer $m$.  Hence $b_n^m=I_n$ only when $m=0$, so $b_n$ has infinite order.

Now consider $a_n$.  Its action on the standard basis is
\[
a_ne_1=\eps_n e_n,\qquad a_ne_j=e_{j-1}\quad (2\le j\le n).
\]
Applying this $n$ times sends each basis vector back to itself, and exactly one wrap-around
occurs.  Therefore
\[
a_n^n=\eps_n I_n.
\]
If $n$ is odd, then $\eps_n=1$, so $a_n^n=I_n$.  No smaller positive power can be the
identity, because $a_n$ permutes the basis cyclically and a proper positive divisor of $n$
does not return every basis vector to itself.  Hence $a_n$ has order $n$.

If $n$ is even, then $\eps_n=-1$, so $a_n^n=-I_n$, and therefore $a_n^{2n}=I_n$.  Again,
no smaller positive power can be the identity. Indeed, if $a_n^m=I_n$, then the
underlying cyclic permutation of the basis forces $n\mid m$. For $m<n$ the basis vectors are
not returned to their original positions, while for $m=n$ one gets $-I_n$, not $I_n$.
Hence the next possible value is $m=2n$, and therefore $a_n$ has order $2n$.
\end{proof}

\subsection{Conjugating the basic transvection}

Let $e_1,\dots,e_n$ be the standard basis of $\ZZ^n$.  By definition,
\[
ae_1=\eps_n e_n,\qquad ae_j=e_{j-1}\quad (2\le j\le n).
\]
Equivalently,
\[
a^{-1}e_n=\eps_n e_1,\qquad a^{-1}e_j=e_{j+1}\quad (1\le j\le n-1).
\]
Recall that a monomial matrix is a square matrix with exactly one nonzero entry in each
row and exactly one nonzero entry in each column. Equivalently, it is a permutation
matrix with nonzero scalars replacing the \(1\)'s. Thus a monomial matrix \(g\) acts on
the standard basis by \(g e_r=\lambda_r e_{\pi(r)}\), where \(\pi\) is a permutation and
each \(\lambda_r\) is nonzero.

\begin{lemma}\label{lem:shift}
For $0\le k\le n-2$ one has
\[
a^{-k}E_{2,1}a^k=E_{k+2,k+1},
\qquad\text{hence}\qquad
a^{-k}ba^k=I_n+E_{k+2,k+1}=T_{k+2,k+1}.
\]
Moreover
\[
a^{-(n-1)}E_{2,1}a^{n-1}=\eps_n E_{1n},
\qquad\text{hence}\qquad
a^{-(n-1)}ba^{n-1}=I_n+\eps_n E_{1n}.
\]
\end{lemma}

\begin{proof}
We first note the following elementary observation. Suppose \(g\) is a monomial
matrix and
\[
g e_r=\lambda_r e_{\pi(r)}
\qquad (1\le r\le n),
\]
where \(\pi\) is a permutation of \(\{1,\dots,n\}\) and each \(\lambda_r\) is a nonzero
scalar. Then
\[
g^{-1}E_{ij}g
=
\lambda_{\pi^{-1}(j)}\lambda_{\pi^{-1}(i)}^{-1}
E_{\pi^{-1}(i),\pi^{-1}(j)}.
\]
Indeed, it is enough to check both sides on the standard basis. For a basis vector
\(e_r\),
\[
g^{-1}E_{ij}g e_r
=
g^{-1}E_{ij}(\lambda_r e_{\pi(r)}).
\]
This is zero unless \(\pi(r)=j\), i.e. unless \(r=\pi^{-1}(j)\). In that exceptional case,
\[
g^{-1}E_{ij}g e_{\pi^{-1}(j)}
=
g^{-1}\bigl(\lambda_{\pi^{-1}(j)} e_i\bigr).
\]
Since
\[
g e_{\pi^{-1}(i)}
=
\lambda_{\pi^{-1}(i)}e_i,
\]
we have
\[
g^{-1}e_i
=
\lambda_{\pi^{-1}(i)}^{-1}e_{\pi^{-1}(i)}.
\]
Therefore
\[
g^{-1}E_{ij}g e_{\pi^{-1}(j)}
=
\lambda_{\pi^{-1}(j)}\lambda_{\pi^{-1}(i)}^{-1}
e_{\pi^{-1}(i)}.
\]
This is exactly the action of
\[
\lambda_{\pi^{-1}(j)}\lambda_{\pi^{-1}(i)}^{-1}
E_{\pi^{-1}(i),\pi^{-1}(j)}.
\]

Apply this with \(g=a^k\). For \(0\le k\le n-2\), one has
\[
a^ke_{k+1}=e_1,
\qquad
a^ke_{k+2}=e_2,
\]
and both relevant scalar coefficients are equal to \(1\). Hence
\[
a^{-k}E_{2,1}a^k=E_{k+2,k+1}.
\]
For \(k=n-1\), one has
\[
a^{n-1}e_n=e_1,
\qquad
a^{n-1}e_1=\eps_n e_2.
\]
Thus the same formula gives
\[
a^{-(n-1)}E_{2,1}a^{n-1}=\eps_n E_{1n}.
\]
The formulas for \(b=I_n+E_{2,1}\) follow immediately.
\end{proof}

The next lemma is used only in situations where the indicated shift does not wrap around
the cycle.

\begin{lemma}\label{lem:shift-general}
Let $1\le p\neq q\le n$, and let $k\ge 0$. Assume that
\[
p+k\le n,\qquad q+k\le n.
\]
Then
\[
a^{-k}E_{p,q}a^k=E_{p+k,q+k},
\qquad\text{hence}\qquad
a^{-k}T_{p,q}a^k=T_{p+k,q+k}.
\]
\end{lemma}

\begin{proof}
We use the monomial-conjugation formula established in the proof of
Lemma~\ref{lem:shift}. Since no wrap-around occurs under the stated hypotheses, one has
\[
a^k e_{p+k}=e_p,\qquad a^k e_{q+k}=e_q,
\]
and both relevant coefficients are equal to $1$. Therefore
\[
a^{-k}E_{p,q}a^k=E_{p+k,q+k}.
\]
The corresponding identity for transvections follows immediately from
$T_{p,q}=I_n+E_{p,q}$.
\end{proof}

\subsection{Recursive words for all transvections}

We now define words $\tau_{ij}(a,b)$ recursively.

\medskip
\noindent
\textbf{Step 1: the adjacent lower transvections.}
For $1\le r\le n-1$, define
\[
\tau_{r+1,r}:=a^{-(r-1)}ba^{r-1}.
\]
By Lemma~\ref{lem:shift}, this is exactly $T_{r+1,r}$.

Also define
\[
\tau_{1n}:=
\bigl(a^{-(n-1)}ba^{n-1}\bigr)^{\eps_n}.
\]
Since $\eps_n=\pm 1$, this means that
\[
\tau_{1n}=
\begin{cases}
a^{-(n-1)}ba^{n-1}, & \text{if }n\text{ is odd},\\[4pt]
\bigl(a^{-(n-1)}ba^{n-1}\bigr)^{-1}, & \text{if }n\text{ is even}.
\end{cases}
\]
Hence $\tau_{1n}=T_{1n}$.

\medskip
\noindent
\textbf{Step 2: all lower transvections.}
For each distance $d=2,\dots,n-1$, define all words $\tau_{ij}$ with $i-j=d$ by
\[
\tau_{ij}:=\comm{\tau_{i,j+1}}{\tau_{j+1,j}},
\]
so that the right-hand side involves only lower transvections of smaller distance.
Since in $\SL_n(\ZZ)$ one has
\[
[T_{i,j+1},T_{j+1,j}]=T_{ij},
\]
it follows inductively that $\tau_{ij}=T_{ij}$ for every $i>j$.

\medskip
\noindent
\textbf{Step 3: the last-column upper transvections.}
For $2\le i\le n-1$, define
\[
\tau_{in}:=\comm{\tau_{i1}}{\tau_{1n}}.
\]
Then
\[
\tau_{in}=[T_{i1},T_{1n}]=T_{in}.
\]
We already know $\tau_{1n}=T_{1n}$ from Step~1.

\medskip
\noindent
\textbf{Step 4: all remaining upper transvections.}
For $1\le i<j\le n-1$, define
\[
\tau_{ij}:=\comm{\tau_{in}}{\tau_{nj}}.
\]
Here $\tau_{1n}$ is the word already defined in Step~1.
Since $\tau_{in}=T_{in}$ and $\tau_{nj}=T_{nj}$, we get
\[
\tau_{ij}=[T_{in},T_{nj}]=T_{ij}.
\]

We have proved:

\begin{theorem}\label{thm:all-transvections}
For every pair of distinct indices $i,j\in\{1,\dots,n\}$, the recursively defined word
$\tau_{ij}(a,b)$ is exactly the elementary transvection $T_{ij}$.  In particular,
$a,b$ generate $\SL_n(\ZZ)$.
\end{theorem}

\begin{proof}
Each step was justified directly by a Steinberg commutator identity.  Since the
transvections generate $\SL_n(\ZZ)$, the last statement follows.
\end{proof}

\subsection{Expressing \texorpdfstring{$a_n$}{a n} back in terms of transvections}\label{sec:presentations}

To apply Proposition~\ref{prop:tietze-template}, it remains to express the generators
$a_n$ and $b_n$ back in terms of the elementary transvections.

Here and throughout, matrices act on column vectors on the left.

For $1\le r\le n-1$, define
\[
\sigma_r:=T_{r,r+1}T_{r+1,r}^{-1}T_{r,r+1}.
\]
On the $(r,r+1)$-block, this is the matrix
\[
\begin{pmatrix}0&1\\ -1&0\end{pmatrix},
\]
and it acts as the identity on all other basis vectors.

\begin{lemma}\label{lem:a-product}
One has
\[
a_n=\sigma_{n-1}\sigma_{n-2}\cdots \sigma_1.
\]
\end{lemma}

\begin{proof}
Each $\sigma_r$ acts on the standard basis by
\[
\sigma_r e_r=-e_{r+1},\qquad \sigma_r e_{r+1}=e_r,\qquad
\sigma_r e_s=e_s\quad (s\neq r,r+1).
\]
We claim that for each $m$ with $1\le m\le n-1$,
\[
(\sigma_m\cdots \sigma_1)e_1=(-1)^m e_{m+1},
\]
\[
(\sigma_m\cdots \sigma_1)e_j=e_{j-1}\qquad (2\le j\le m+1),
\]
and
\[
(\sigma_m\cdots \sigma_1)e_j=e_j\qquad (j\ge m+2).
\]
This is proved by induction on $m$. The case $m=1$ is immediate. For the induction
step, note that $\sigma_{m+1}$ acts nontrivially only on $e_{m+1}$ and $e_{m+2}$, so
the claimed formulas transform exactly as required.

Taking $m=n-1$, one obtains
\[
(\sigma_{n-1}\cdots \sigma_1)e_1=(-1)^{n-1}e_n=\eps_n e_n,
\]
and
\[
(\sigma_{n-1}\cdots \sigma_1)e_j=e_{j-1}\qquad (2\le j\le n).
\]
This is precisely the action of $a_n$ on the standard basis, so the two matrices are equal.
\end{proof}

By Theorem~\ref{thm:all-transvections}, each $T_{i,j}$ may be replaced by the word
$\tau_{i,j}(a,b)$. Lemma~\ref{lem:a-product} therefore yields the explicit identity
\begin{equation}\label{eq:a-word}
a=
\bigl(\tau_{n-1,n}\tau_{n,n-1}^{-1}\tau_{n-1,n}\bigr)
\cdots
\bigl(\tau_{1,2}\tau_{2,1}^{-1}\tau_{1,2}\bigr).
\end{equation}
Also, by construction,
\[
b=\tau_{2,1}.
\]

We may now combine the Steinberg presentation with Proposition~\ref{prop:tietze-template}.

\begin{theorem}\label{thm:uniform-presentation}
For every $n\ge 3$, the group $\SL_n(\ZZ)$ admits the finite $2$-generator presentation
\[
\SL_n(\ZZ)\cong \langle a,b\mid \mathcal R_n(a,b)\rangle,
\]
where $\mathcal R_n(a,b)$ consists of the following relations:
\begin{align*}
[\tau_{i,j}(a,b),\tau_{k,l}(a,b)]&=1
&&\bigl((i,j)\neq (k,l),\ i\neq l,\ j\neq k\bigr),\\
[\tau_{i,j}(a,b),\tau_{j,k}(a,b)]&=\tau_{i,k}(a,b)
&&\text{for distinct }i,j,k,\\
\bigl(\tau_{1,2}(a,b)\tau_{2,1}(a,b)^{-1}\tau_{1,2}(a,b)\bigr)^4
&=1,\\
a&=\bigl(\tau_{n-1,n}\tau_{n,n-1}^{-1}\tau_{n-1,n}\bigr)\cdots{}\\
&\qquad\bigl(\tau_{1,2}\tau_{2,1}^{-1}\tau_{1,2}\bigr).
\end{align*}
The relation $b=\tau_{2,1}(a,b)$ is omitted because it is tautological. The first family
of commutativity relations is listed without identifying inverse-duplicate commutativity relators; in Section~\ref{sec:quotients} we keep
only one orientation from each inverse pair when counting relators.
\end{theorem}

\begin{proof}
Apply Proposition~\ref{prop:tietze-template} to the Steinberg presentation
\eqref{eq:steinberg}, taking the generators $x_{i,j}$ to be the transvections $T_{i,j}$,
the words $w_{i,j}(a,b)$ to be $\tau_{i,j}(a,b)$, the word $u$ to be
$\sigma_{n-1}\cdots \sigma_1$, and the word $v$ to be $T_{2,1}$. The required
hypotheses hold by Theorem~\ref{thm:all-transvections}, Lemma~\ref{lem:a-product},
and the identity $b=\tau_{2,1}$. After substitution, the $v$-relation reads
$b=\tau_{2,1}(a,b)=b$, so it may be omitted. The displayed presentation is exactly the
non-tautological part of the presentation produced by Proposition~\ref{prop:tietze-template}.
\end{proof}

\begin{remark}
The formal passage from the transvection presentation to a two-generator
presentation is precisely the standard Tietze-elimination step recalled in
Proposition~\ref{prop:tietze-template}. The new ingredient in
Theorem~\ref{thm:uniform-presentation} is the explicit all-rank construction
of the words \(\tau_{ij}(a,b)\) for the particular monomial/transvection pair
\((a_n,b_n)\).
\end{remark}

\begin{remark}
The presentation of Theorem~\ref{thm:uniform-presentation} is not intended to minimize length.
Its main feature is that it is completely explicit, uniform in $n$, and obtained directly
from the Steinberg presentation by a transparent elimination procedure.
\end{remark}

\begin{remark}\label{rem:quantitative-preview}
The recursive words $\tau_{ij}(a,b)$ used above are completely explicit, but they are
unbalanced. In Section~\ref{sec:quantitative} we replace them by a balanced
family of words representing the same transvections and obtain polynomial bounds for the
word lengths and for the size of the resulting $2$-generator presentation.
\end{remark}

\section{Quantitative complexity of the transvection words}\label{sec:quantitative}

In this section we rebalance the construction and obtain
polynomial bounds for the lengths of the transvection words, and hence for the size of
the resulting $2$-generator presentation.

Throughout, for a word $w$ in the alphabet $\{a^{\pm1},b^{\pm1}\}$, we write $\ell(w)$
for its freely reduced length.

\subsection{Balanced words for the first-column transvections}

For $2\le r\le n$, we define words $\omega_r(a,b)$ representing the transvections
$T_{r,1}$.

Set
\[
\omega_2:=b.
\]
For $3\le r\le n$, define
\[
m_r:=\Bigl\lfloor \frac r2\Bigr\rfloor+1,
\qquad
d_r:=r-m_r+1=\Bigl\lceil \frac r2\Bigr\rceil,
\]
and then set recursively
\[
\omega_r:=
\comm{\,a^{-(m_r-1)}\omega_{d_r}a^{m_r-1}\,}{\omega_{m_r}}.
\]

\begin{lemma}\label{lem:balanced-first-column}
For every $r$ with $2\le r\le n$, one has
\[
\omega_r(a,b)=T_{r,1}.
\]
\end{lemma}

\begin{proof}
We argue by induction on $r$.

For $r=2$, this is the definition $\omega_2=b=T_{2,1}$.

Now let $r\ge 3$ and assume inductively that $\omega_s=T_{s,1}$ for all $2\le s<r$.
By construction, $2\le d_r,m_r<r$, so the inductive hypothesis gives
\[
\omega_{d_r}=T_{d_r,1},
\qquad
\omega_{m_r}=T_{m_r,1}.
\]
Since
\[
d_r+(m_r-1)=r,
\]
and no wrap-around occurs in this conjugation, Lemma~\ref{lem:shift-general} gives
\[
a^{-(m_r-1)}\omega_{d_r}a^{m_r-1}
=
a^{-(m_r-1)}T_{d_r,1}a^{m_r-1}
=
T_{r,m_r}.
\]
Therefore
\[
\omega_r
=
[T_{r,m_r},T_{m_r,1}]
=
T_{r,1}
\]
by the Steinberg commutator relation.
\end{proof}

\subsection{Length bounds for the first-column words}

Let
\[
L_r:=\ell(\omega_r)\qquad (2\le r\le n),
\]
and
\[
M_r:=\max_{2\le s\le r}L_s.
\]

\begin{lemma}\label{lem:first-column-length-recurrence}
For every $3\le r\le n$,
\[
L_r\le 2L_{d_r}+2L_{m_r}+4(m_r-1),
\]
and hence
\[
M_r\le 4M_{\lceil r/2\rceil+1}+2r.
\]
\end{lemma}

\begin{proof}
The first inequality is immediate from the definition of $\omega_r$:
\[
\omega_r=[u,v]
\quad\text{with}\quad
u=a^{-(m_r-1)}\omega_{d_r}a^{m_r-1},
\quad
v=\omega_{m_r}.
\]
Since $\ell([u,v])\le 2\ell(u)+2\ell(v)$, we obtain
\[
L_r\le 2\bigl(L_{d_r}+2(m_r-1)\bigr)+2L_{m_r}
=2L_{d_r}+2L_{m_r}+4(m_r-1).
\]
Now $d_r=\lceil r/2\rceil$ and $m_r\le \lceil r/2\rceil+1$, while
\[
4(m_r-1)\le 2r.
\]
Hence
\[
L_r\le 4M_{\lceil r/2\rceil+1}+2r.
\]
For every $s\le r$, if $s\le \lceil r/2\rceil+1$, then
\[
L_s\le M_{\lceil r/2\rceil+1}\le 4M_{\lceil r/2\rceil+1}+2r.
\]
If $s>\lceil r/2\rceil+1$, then applying the preceding bound to $L_s$ gives
\[
L_s\le 4M_{\lceil s/2\rceil+1}+2s
\le 4M_{\lceil r/2\rceil+1}+2r.
\]
Taking the maximum over $2\le s\le r$ gives the asserted bound for $M_r$.
\end{proof}

\begin{proposition}\label{prop:first-column-quadratic}
There exists an absolute constant $C_0>0$ such that
\[
L_r\le C_0\,r^2
\qquad\text{for all }2\le r\le n.
\]
\end{proposition}

\begin{proof}
Define
\[
P_t:=M_{\min(n,2^t+2)}\qquad (t\ge 0).
\]
We claim that, for $t\ge 1$,
\[
P_t\le 4P_{t-1}+2^{t+1}+4.
\]
Indeed, let $2\le s\le \min(n,2^t+2)$. If $s\le 2^{t-1}+2$, then
\[
L_s\le P_{t-1}.
\]
If $s>2^{t-1}+2$, then Lemma~\ref{lem:first-column-length-recurrence} gives
\[
L_s\le 4M_{\lceil s/2\rceil+1}+2s.
\]
Since $s\le 2^t+2$, we have
\[
\Bigl\lceil \frac s2\Bigr\rceil+1\le 2^{t-1}+2.
\]
Also $\lceil s/2\rceil+1\le s\le n$. Hence
\[
M_{\lceil s/2\rceil+1}\le P_{t-1},
\]
and therefore
\[
L_s\le 4P_{t-1}+2^{t+1}+4.
\]
Taking the maximum over all such $s$ proves the claimed recurrence.

Iterating this recurrence yields
\[
P_t
\le
4^tP_0+\sum_{u=1}^t4^{t-u}(2^{u+1}+4).
\]
Now $P_0=M_3$, and $M_3$ is bounded by an absolute constant. Moreover,
\[
\sum_{u=1}^t4^{t-u}2^{u+1}
=
2\cdot4^t\sum_{u=1}^t2^{-u}
<
2\cdot4^t,
\]
and
\[
\sum_{u=1}^t4^{t-u}\cdot4
=
4\sum_{u=1}^t4^{t-u}
<
\frac{4}{3}\,4^t.
\]
Hence there exists an absolute constant $C_0'>0$ such that
\[
P_t\le C_0'4^t
\qquad(t\ge 0).
\]
Now let $2\le r\le n$, and choose $t$ minimal with $r\le 2^t+2$. Then
\[
L_r\le M_r\le P_t\le C_0'4^t.
\]
For $t\ge 1$, the minimality of $t$ implies
\[
2^{t-1}+2<r,
\]
so $4^t\le 16r^2$. Hence
\[
L_r\le 16C_0'\,r^2.
\]
Taking $C_0:=16C_0'$, and increasing it if necessary to handle the finitely many
small cases, proves the proposition.
\end{proof}

\subsection{All lower and upper transvections}

For $i>j$, define
\[
\widetilde\tau_{ij}:=a^{-(j-1)}\omega_{\,i-j+1}a^{j-1}.
\]

\begin{lemma}\label{lem:all-lower-quadratic}
For all $i>j$,
\[
\widetilde\tau_{ij}=T_{i,j},
\]
and there exists an absolute constant $C_1>0$ such that
\[
\ell(\widetilde\tau_{ij})\le C_1 n^2.
\]
\end{lemma}

\begin{proof}
By Lemma~\ref{lem:balanced-first-column},
\[
\omega_{\,i-j+1}=T_{\,i-j+1,\,1}.
\]
Since no wrap-around occurs in the conjugation by $a^{j-1}$, Lemma~\ref{lem:shift-general}
gives
\[
a^{-(j-1)}T_{\,i-j+1,\,1}a^{j-1}=T_{i,j}.
\]
Thus $\widetilde\tau_{ij}=T_{i,j}$.

Also,
\[
\ell(\widetilde\tau_{ij})
\le
L_{\,i-j+1}+2(j-1)
\le C_0(i-j+1)^2+2n
\le C_1n^2
\]
for a suitable absolute constant $C_1$.
\end{proof}

Next define
\[
\rho_n:=\tau_{1n}
=
\bigl(a^{-(n-1)}ba^{n-1}\bigr)^{\eps_n},
\]
and for $2\le s\le n-1$ define
\[
\rho_s:=\comm{\rho_n}{\widetilde\tau_{ns}}.
\]

\begin{lemma}\label{lem:first-row-quadratic}
For each $2\le s\le n$, one has
\[
\rho_s=T_{1,s},
\]
and there exists an absolute constant $C_2>0$ such that
\[
\ell(\rho_s)\le C_2 n^2.
\]
\end{lemma}

\begin{proof}
The case $s=n$ is exactly the definition of $\rho_n=\tau_{1n}=T_{1,n}$.

Now let $2\le s\le n-1$. By Lemma~\ref{lem:all-lower-quadratic},
\[
\widetilde\tau_{ns}=T_{n,s}.
\]
Hence
\[
\rho_s=[T_{1,n},T_{n,s}]=T_{1,s}
\]
by the Steinberg commutator relation.

For the length bound, $\ell(\rho_n)=2n-1$, while
\[
\ell(\rho_s)\le 2\ell(\rho_n)+2\ell(\widetilde\tau_{ns})
\le 2(2n-1)+2C_1n^2
\le C_2n^2
\]
for a suitable absolute constant $C_2$.
\end{proof}

Finally, for $i<j$, define
\[
\widetilde\tau_{ij}:=a^{-(i-1)}\rho_{\,j-i+1}a^{i-1}.
\]

\begin{theorem}\label{thm:quadratic-transvections}
There exists an absolute constant $C>0$ such that for every $n\ge 3$ and every pair of
distinct indices $i,j\in\{1,\dots,n\}$, the transvection $T_{i,j}$ is represented by a
word $\widetilde\tau_{ij}(a,b)$ in the alphabet $\{a^{\pm1},b^{\pm1}\}$ satisfying
\[
\ell(\widetilde\tau_{ij})\le Cn^2.
\]
\end{theorem}

\begin{proof}
The lower-triangular case $i>j$ is Lemma~\ref{lem:all-lower-quadratic}. For $i<j$, one
has $\rho_{j-i+1}=T_{1,\,j-i+1}$ by Lemma~\ref{lem:first-row-quadratic}, and again no
wrap-around occurs in the conjugation by $a^{i-1}$, so Lemma~\ref{lem:shift-general}
gives
\[
a^{-(i-1)}T_{1,\,j-i+1}a^{i-1}=T_{i,j}.
\]
Thus $\widetilde\tau_{ij}=T_{i,j}$.

The length bound is immediate from Lemmas~\ref{lem:all-lower-quadratic} and
\ref{lem:first-row-quadratic}, since the additional conjugation contributes at most
$2(i-1)\le 2n$.
\end{proof}

\subsection{A polynomial-size \texorpdfstring{$2$}{2}-generator presentation}

Applying Proposition~\ref{prop:tietze-template} with the balanced words
$\widetilde\tau_{ij}(a,b)$ in place of the original words $\tau_{ij}(a,b)$ gives another
explicit $2$-generator presentation of $\SL_n(\ZZ)$, now with controlled size. In this
presentation, every occurrence of $\tau_{ij}$ in Theorem~\ref{thm:uniform-presentation}
is replaced by $\widetilde\tau_{ij}$.

\begin{corollary}\label{cor:polynomial-presentation}
There exist absolute constants $C_3,C_4>0$ such that for every $n\ge 3$, the group
$\SL_n(\ZZ)$ admits a finite presentation on two generators with at most $C_3n^4$
defining relators and total relator length at most $C_4n^6$.
\end{corollary}

\begin{proof}
Use the Steinberg presentation \eqref{eq:steinberg} and Proposition~\ref{prop:tietze-template},
with the transvections represented by the words $\widetilde\tau_{ij}(a,b)$ of
Theorem~\ref{thm:quadratic-transvections}. The total relator length is measured in the
free group on $\{a,b\}$, after each displayed equation $U=V$ is interpreted as the
relator $UV^{-1}$.

There are $O(n^4)$ commutativity relations of the form
\[
[\widetilde\tau_{ij},\widetilde\tau_{kl}]=1,
\]
and each has length $O(n^2)$, since each $\widetilde\tau_{ij}$ has length $O(n^2)$.

There are $O(n^3)$ Steinberg commutator relations of the form
\[
[\widetilde\tau_{ij},\widetilde\tau_{jk}]=\widetilde\tau_{ik},
\]
and each again has length $O(n^2)$.

The torsion relation
\[
(\widetilde\tau_{12}\widetilde\tau_{21}^{-1}\widetilde\tau_{12})^4=1
\]
has length $O(n^2)$.

Finally, the relation expressing $a$ as
\[
a=
(\widetilde\tau_{n-1,n}\widetilde\tau_{n,n-1}^{-1}\widetilde\tau_{n-1,n})
\cdots
(\widetilde\tau_{1,2}\widetilde\tau_{2,1}^{-1}\widetilde\tau_{1,2})
\]
has length $O(n^3)$, since it is a product of $n-1$ blocks and each block has length
$O(n^2)$.

Thus the number of relators is $O(n^4)$ and the total relator length is
\[
O(n^4)\cdot O(n^2)+O(n^3)=O(n^6),
\]
as claimed.
\end{proof}

\begin{remark}
Corollary~\ref{cor:polynomial-presentation} shows that once one asks not merely
for an explicit all-rank $2$-generator presentation, but also for uniform control on the
lengths of the transvection words and on the resulting presentation size, the problem is
no longer formal. The balanced first-column recursion above provides such control.
\end{remark}

\section{Low-rank remarks and uniform variants}\label{sec:low-rank-variants}

\subsection{The case \texorpdfstring{$n=3$}{n = 3}}\label{sec:n3}

Conder--Liversidge--Vsemirnov give eight explicit finite $2$-generator presentations for
$\SL_3(\ZZ)$ \cite{Conder}, building on the earlier $3$-generator presentation of
Conder--Robertson--Williams \cite{ConderRW}. Our uniform construction specialises in
rank \(3\) to an explicit \(2\)-generator presentation. Their eighth presentation,
labelled \textup{(h)}, is based on Trott's pair, and for \(n=3\) that pair coincides
with \((a_3,b_3)\). Thus the rank-\(3\) point here is not a new existence statement,
but that Trott's pair and the standard presentation-theoretic mechanism recalled in
Proposition~\ref{prop:tietze-template} arise as the first instance of the uniform
all-rank construction, together with the quantitative bounds proved in
Section~\ref{sec:quantitative}.

\subsection{Further uniform variants of the basic presentation}\label{sec:variants}

The presentation of Theorem~\ref{thm:uniform-presentation} is built from the pair
$(a_n,b_n)$, where $a_n$ has finite order and $b_n$ has infinite order. Because the
construction is explicit, one can make further explicit changes of generators and obtain
new uniform finite presentations on pairs with prescribed order type. We note two such
variants here.

\subsubsection{A uniform infinite--infinite variant}

We now make an explicit change of generators that yields a pair of infinite-order
generators.

\begin{proposition}\label{prop:infinite-pair}
For every \(n\ge 3\), both
\[
x_n:=a_nb_n,\qquad y_n:=b_n
\]
have infinite order, and
\[
a_n=x_ny_n^{-1},\qquad b_n=y_n.
\]
Hence
\[
\SL_n(\ZZ)\cong \bigl\langle x,y \,\big|\, \mathcal R_n(xy^{-1},y)\bigr\rangle .
\]
\end{proposition}

\begin{proof}
Since \(y_n=b_n=I_n+E_{2,1}\), Proposition~\ref{prop:orders} shows that \(y_n\) has
infinite order.

It remains to prove that \(x_n=a_nb_n\) has infinite order. We use the following
criterion of Conder--Liversidge--Vsemirnov \cite[Proposition~3.1]{Conder}: if
\(M\in \SL_n(\QQ)\) has finite order, then
\[
\tr(M)=\tr(M^{-1}).
\]

Now
\[
x_n=a_nb_n=a_n(I_n+E_{2,1})=a_n+a_nE_{2,1}.
\]
Because \(a_ne_2=e_1\), one has \(a_nE_{2,1}=E_{1,1}\). Hence
\[
x_n=
\begin{pmatrix}
1&1&0&\cdots&0\\
0&0&1&\ddots&\vdots\\
\vdots&\ddots&\ddots&\ddots&0\\
0&\cdots&0&0&1\\
\eps_n&0&\cdots&0&0
\end{pmatrix}.
\]
In particular,
\[
\tr(x_n)=1.
\]

On the other hand,
\[
x_n^{-1}=b_n^{-1}a_n^{-1}=(I_n-E_{2,1})a_n^{-1}.
\]
The matrix \(a_n^{-1}\) has zero diagonal, since it is a monomial matrix whose underlying
permutation is an \(n\)-cycle. Also \(E_{2,1}a_n^{-1}\) has nonzero entries only in row
\(2\), and in fact only possibly in column \(n\), so it has zero trace for \(n\ge 3\).
Therefore
\[
\tr(x_n^{-1})=0.
\]

Thus
\[
\tr(x_n)\neq \tr(x_n^{-1}).
\]
By the criterion quoted above, \(x_n\) cannot have finite order. Hence \(x_n\) has
infinite order.

Finally,
\[
a_n=x_ny_n^{-1},\qquad b_n=y_n
\]
are immediate from the definitions. Substituting these into the relators
\(\mathcal R_n(a,b)\) of Theorem~\ref{thm:uniform-presentation} yields the displayed
presentation.
\end{proof}

\begin{remark}
Generation of $\SL_n(\ZZ)$ by two elements of infinite order is known in general; see,
for example, \cite{Conder}.
The point here is that the explicit uniform presentation of
Theorem~\ref{thm:uniform-presentation} also admits a completely explicit all-rank
infinite--infinite variant.
\end{remark}

\subsubsection{A uniform finite--finite variant}

We now construct a second uniform change of generators giving a pair of finite order.

Define
\[
u_n:=a_n,\qquad
v_n:=b_n^{-1}T_{1,2}.
\]
Since $b_n^{-1}=I_n-E_{2,1}$ and $T_{1,2}=I_n+E_{1,2}$, one has
\[
v_n=
\begin{pmatrix}
1&1\\
-1&0
\end{pmatrix}\oplus I_{n-2}.
\]

\begin{lemma}\label{lem:finite-pair-bridge}
For every $n\ge 3$, the element $u_n$ has finite order $n$ if $n$ is odd and $2n$ if
$n$ is even, while $v_n$ has order $6$. Moreover, if
\[
W(u,v):=(uvu^{-1})v^{-1}(uv^{-1}u^{-1})v^{-1}(uvu^{-1})v,
\]
then
\[
W(u_n,v_n)=b_n.
\]
\end{lemma}

\begin{proof}
The order of $u_n=a_n$ is given by Proposition~\ref{prop:orders}.

For $v_n$, the nontrivial block is
\[
C=\begin{pmatrix}1&1\\ -1&0\end{pmatrix},
\]
and one checks directly that
\[
C^3=-I_2.
\]
Hence $C$ has order $6$, so $v_n$ has order $6$.

It remains to prove the identity $W(u_n,v_n)=b_n$. Set
\[
c_n:=u_nv_nu_n^{-1},\qquad d_n:=u_nv_n^{-1}u_n^{-1}.
\]
Since $v_n$ acts nontrivially only on the span of $e_1,e_2$, while $u_n=a_n$ cyclically
permutes the basis vectors, the matrices $c_n$ and $d_n$ act nontrivially only on the span
of $e_1$ and $e_n$. Therefore the word
\[
W(u_n,v_n)=c_nv_n^{-1}d_nv_n^{-1}c_nv_n
\]
acts trivially on the subspace
\[
U^\perp:=\langle e_3,\dots,e_{n-1}\rangle.
\]
So it suffices to compute its action on
\[
U:=\langle e_1,e_2,e_n\rangle.
\]

With respect to the ordered basis $(e_1,e_2,e_n)$ of $U$, one has
\[
v_n|_U=
\begin{pmatrix}
1&1&0\\
-1&0&0\\
0&0&1
\end{pmatrix},
\qquad
c_n|_U=
\begin{pmatrix}
0&0&-\eps_n\\
0&1&0\\
\eps_n&0&1
\end{pmatrix},
\qquad
d_n|_U=
\begin{pmatrix}
1&0&\eps_n\\
0&1&0\\
-\eps_n&0&0
\end{pmatrix}.
\]
Also,
\[
v_n^{-1}|_U=
\begin{pmatrix}
0&-1&0\\
1&1&0\\
0&0&1
\end{pmatrix}.
\]
Multiplying step by step and using $\eps_n^2=1$, one finds
\[
(c_nv_n^{-1})\big|_U=
\begin{pmatrix}
0&0&-\eps_n\\
1&1&0\\
0&-\eps_n&1
\end{pmatrix},
\qquad
(d_nv_n^{-1})\big|_U=
\begin{pmatrix}
0&-1&\eps_n\\
1&1&0\\
0&\eps_n&0
\end{pmatrix},
\]
and hence
\[
(c_nv_n^{-1}d_nv_n^{-1})\big|_U=
\begin{pmatrix}
0&-1&0\\
1&0&\eps_n\\
-\eps_n&0&0
\end{pmatrix}.
\]
Multiplying by $(c_nv_n)|_U$ on the right now gives
\[
(c_nv_n^{-1}d_nv_n^{-1}c_nv_n)\big|_U=
\begin{pmatrix}
1&0&0\\
1&1&0\\
0&0&1
\end{pmatrix}.
\]
This is exactly the restriction of $b_n=I_n+E_{2,1}$ to $U$, and on $U^\perp$ both
matrices are the identity. Therefore
\[
W(u_n,v_n)=b_n.
\]
\end{proof}

\begin{theorem}\label{thm:finite-pair-presentation}
For every $n\ge 3$, let
\[
u_n:=a_n,\qquad v_n:=b_n^{-1}T_{1,2}.
\]
Then $u_n$ has order $n$ if $n$ is odd and order $2n$ if $n$ is even, while $v_n$ has
order $6$. Moreover, $\SL_n(\ZZ)$ admits a finite presentation on generators $u,v$
corresponding to the pair $(u_n,v_n)$:
\[
\SL_n(\ZZ)\cong
\left\langle u,v \,\middle|\,
\mathcal R_n\bigl(u,W(u,v)\bigr),
v=W(u,v)^{-1}\tau_{1,2}\bigl(u,W(u,v)\bigr)
\right\rangle.
\]
Here $\tau_{1,2}$ denotes the recursive word introduced in
Theorem~\ref{thm:all-transvections}; one may equally use $\widetilde\tau_{1,2}$.
\end{theorem}

\begin{proof}
Introduce new generators $u,v$ with
\[
u=a,\qquad v=b^{-1}\tau_{1,2}(a,b).
\]
For the distinguished pair $(u_n,v_n)$ this gives
\[
u_n=a_n,\qquad v_n=b_n^{-1}T_{1,2}=b_n^{-1}\tau_{1,2}(a_n,b_n).
\]
By Lemma~\ref{lem:finite-pair-bridge}, the inverse substitution is
\[
a=u,\qquad b=W(u,v).
\]
Applying Proposition~\ref{prop:tietze-template} to
Theorem~\ref{thm:uniform-presentation} with these substitutions yields the displayed
presentation, after omitting the tautological relation $u=u$.
\end{proof}

\begin{remark}
If one wants the finite orders to be visible in the presentation, one may adjoin the
redundant relators
\[
u^n=1\quad(n\text{ odd}),\qquad u^{2n}=1\quad(n\text{ even}),\qquad v^6=1.
\]
\end{remark}

\begin{remark}
Again, generation of $\SL_n(\ZZ)$ by two finite-order elements is known in many forms;
see, for example, \cite{Conder}.
The point here is that the uniform presentation on $(a_n,b_n)$ also yields a completely
explicit all-rank finite--finite presentation family.
\end{remark}

\section{Quotients and relator counts}\label{sec:quotients}

The uniform pair $(a_n,b_n)$ has two further advantages that do not depend on any
low-rank calculations: it behaves well under reduction modulo $m$, and in even rank it
descends immediately to an explicit presentation of $\PSL_n(\ZZ)$.

\subsection{Reduction modulo \texorpdfstring{$m$}{m}}

Fix $n\ge 3$ and let $\rho_m:\SL_n(\ZZ)\to \SL_n(\ZZ/m\ZZ)$ be reduction modulo
$m\ge 2$. Write
\[
\bar a_n:=\rho_m(a_n),\qquad \bar b_n:=\rho_m(b_n).
\]

\begin{proposition}\label{prop:modm-transvections}
For every $m\ge 2$ and every pair of distinct indices $i,j$, one has
\[
\tau_{ij}(\bar a_n,\bar b_n)=I_n+\bar E_{ij}\in \SL_n(\ZZ/m\ZZ).
\]
In particular, $\bar a_n$ and $\bar b_n$ generate $\SL_n(\ZZ/m\ZZ)$.
\end{proposition}

\begin{proof}
Every word $\tau_{ij}(a,b)$ is built from the generators using multiplication, inverse,
and commutator. Since $\rho_m$ is a group homomorphism, it commutes with all of
these operations. Hence
\[
\tau_{ij}(\bar a_n,\bar b_n)
=\rho_m(\tau_{ij}(a_n,b_n))
=\rho_m(T_{ij})
=I_n+\bar E_{ij},
\]
where the middle equality is Theorem~\ref{thm:all-transvections}.

Thus the subgroup generated by $\bar a_n,\bar b_n$ contains all standard elementary
transvections $I_n+\bar E_{ij}$. For $n\ge 3$, the group $\SL_n(\ZZ/m\ZZ)$ is generated
by elementary matrices, and each elementary matrix has the form
\[
I_n+r\bar E_{ij}=(I_n+\bar E_{ij})^r
\]
because $\bar E_{ij}^{\,2}=0$. Hence the standard transvections $I_n+\bar E_{ij}$ already
generate $\SL_n(\ZZ/m\ZZ)$. Equivalently,
\[
\langle \bar a_n,\bar b_n\rangle
=\rho_m(\langle a_n,b_n\rangle)
=\rho_m(\SL_n(\ZZ))
=\SL_n(\ZZ/m\ZZ). 
\]
\end{proof}

\begin{corollary}\label{cor:modm-orders}
For every $m\ge 2$, the pair $(\bar a_n,\bar b_n)$ explicitly generates
$\SL_n(\ZZ/m\ZZ)$, with
\[
\ord(\bar b_n)=m,
\]
and
\[
\ord(\bar a_n)\mid n \quad\text{if }n\text{ is odd},\qquad
\ord(\bar a_n)\mid 2n \quad\text{if }n\text{ is even}.
\]
\end{corollary}

\begin{proof}
Since $\bar b_n=I_n+\bar E_{21}$ and $\bar E_{21}^2=0$, one has
\[
\bar b_n^r=I_n+r\bar E_{21}
\]
for every integer $r$. Hence $\bar b_n^r=I_n$ if and only if $r\equiv 0\pmod m$, so
$\ord(\bar b_n)=m$.

Also $a_n^n=\eps_n I_n$ in $\SL_n(\ZZ)$ by Proposition~\ref{prop:orders}. Reducing
modulo $m$ gives
\[
\bar a_n^n=\eps_n I_n.
\]
Therefore $\ord(\bar a_n)\mid n$ when $\eps_n=1$ and $\ord(\bar a_n)\mid 2n$ when
$\eps_n=-1$.
\end{proof}

\subsection{An explicit presentation for the projective quotient}

We now use the same presentation to pass to the projective quotient.

\begin{lemma}\label{lem:center}
The centre of $\SL_n(\ZZ)$ is
\[
Z(\SL_n(\ZZ))=
\begin{cases}
\{I_n\}, & n\text{ odd},\\
\{\pm I_n\}, & n\text{ even}.
\end{cases}
\]
\end{lemma}

\begin{proof}
Let $z\in Z(\SL_n(\ZZ))$. Since $z$ commutes with every elementary transvection
$T_{ij}=I_n+E_{ij}$, it also commutes with every matrix unit $E_{ij}$:
\[
zE_{ij}=E_{ij}z\qquad (i\ne j).
\]
Applying both sides to $e_j$ gives
\[
ze_i=zE_{ij}e_j=E_{ij}ze_j.
\]
Since $E_{ij}ze_j$ is always a multiple of $e_i$, it follows that $ze_i\in \ZZ e_i$ for
every $i$, so each coordinate line $\ZZ e_i$ is $z$-stable. Therefore $z$ is diagonal:
say $ze_i=\lambda_i e_i$ with $\lambda_i\in \ZZ$. Applying $zE_{ij}=E_{ij}z$ to $e_j$
again yields
\[
\lambda_i e_i=ze_i=zE_{ij}e_j=E_{ij}ze_j=\lambda_j e_i,
\]
so $\lambda_i=\lambda_j$ for all $i\neq j$. Thus $z=\lambda I_n$ for some integer
$\lambda$. Since $z\in \SL_n(\ZZ)$, one has
$\lambda=\pm 1$, and $\det(z)=\lambda^n=1$. Hence $\lambda=1$ if $n$ is odd, while
$\lambda=\pm1$ if $n$ is even.
\end{proof}

\begin{theorem}\label{thm:psl-presentation}
For odd $n\ge 3$, the quotient map
$\SL_n(\ZZ)\twoheadrightarrow \PSL_n(\ZZ)$ is an isomorphism, so
$\PSL_n(\ZZ)$ has the presentation of Theorem~\ref{thm:uniform-presentation}. For even
$n\ge 4$, one has
\[
\PSL_n(\ZZ)\cong \left\langle a,b\ \middle|\ \mathcal R_n(a,b),\ a^n=1\right\rangle.
\]
\end{theorem}

\begin{proof}
By Lemma~\ref{lem:center}, the centre of $\SL_n(\ZZ)$ is trivial when $n$ is odd, so
the quotient map $\SL_n(\ZZ)\to \PSL_n(\ZZ)$ is an isomorphism.

Now assume $n$ is even. Then Proposition~\ref{prop:orders} gives
\[
a_n^n=-I_n.
\]
By Lemma~\ref{lem:center}, the subgroup generated by $-I_n$ is exactly the centre of
$\SL_n(\ZZ)$. Therefore
\[
\PSL_n(\ZZ)\cong \SL_n(\ZZ)/\langle -I_n\rangle
=\SL_n(\ZZ)/\langle a_n^n\rangle.
\]
Starting from the presentation of Theorem~\ref{thm:uniform-presentation}, adjoining
the additional relation $a^n=1$ therefore kills precisely the centre, and yields the
displayed presentation of $\PSL_n(\ZZ)$.
\end{proof}

\subsection{A quantitative count of relators}

The presentation of Theorem~\ref{thm:uniform-presentation} is explicit but highly
redundant. One can at least count its relators exactly.

\begin{proposition}\label{prop:relator-count}
With the commutativity relators taken up to reversal,
\[
[u,v]=1 \qquad\text{and}\qquad [v,u]=1,
\]
and with the tautological relation $b=\tau_{2,1}(a,b)$ omitted, the presentation of
Theorem~\ref{thm:uniform-presentation} may be written with exactly
\[
\frac12\,n(n+1)(n-1)(n-2)+2
\]
defining relators.
\end{proposition}

\begin{proof}
For a fixed transvection $T_{ij}$, the number of distinct transvections $T_{kl}\neq T_{ij}$
satisfying the Steinberg commutativity condition $i\neq l$ and $j\neq k$ is
\[
(n-1)^2-(n-2)-1=(n-1)(n-2).
\]
Indeed, once $k\neq l$, $i\neq l$, and $j\neq k$ are imposed, the index $k$ may be any
element of $\{1,\dots,n\}\setminus\{j\}$ and the index $l$ may be any element of
$\{1,\dots,n\}\setminus\{i\}$, giving $(n-1)^2$ ordered pairs $(k,l)$. Among these,
one must still exclude the $n-2$ pairs with $k=l\notin\{i,j\}$ and the remaining pair
$(i,j)$ itself.

Equivalently, starting from all $n(n-1)$ ordered pairs $(k,l)$ with $k\neq l$, one
excludes the pair $(i,j)$, the $n-1$ pairs with $k=j$, and the $n-1$ pairs with $l=i$,
then adds back the overlap $(j,i)$. This again gives $(n-1)(n-2)$.

Hence the total number of ordered commuting pairs is
\[
n(n-1)\cdot (n-1)(n-2)=n(n-1)^2(n-2).
\]
The two relators $[T_{ij},T_{kl}]=1$ and $[T_{kl},T_{ij}]=1$ are inverse relators, so one
may keep only one orientation. Thus the number of distinct commutativity relators is
\[
\frac12\,n(n-1)^2(n-2).
\]

Next, the Steinberg relations of the form $[T_{ij},T_{jk}]=T_{ik}$ are indexed by ordered
triples of distinct indices, so there are exactly
\[
n(n-1)(n-2)
\]
such relators.

Finally, there is one torsion relator
\[
(T_{12}T_{21}^{-1}T_{12})^4=1
\]
and one relator expressing $a$ in terms of the transvections. Therefore the total number
of relators is
\[
\frac12\,n(n-1)^2(n-2)+n(n-1)(n-2)+2
=\frac12\,n(n+1)(n-1)(n-2)+2.
\]
\end{proof}

The same relator count applies to the balanced presentation obtained in
Section~\ref{sec:quantitative}, since the number of defining relations is unchanged and
only the chosen words representing the transvections have been replaced.

\subsection{A bounded-girth remark for the fixed-generator congruence family}

For fixed $n\ge 3$, prime $p$, and $k\ge 1$, consider the Cayley graph
\[
X_{n,p^k}:=\Cay\bigl(\SL_n(\ZZ/p^k\ZZ),\{\bar a_{n,p^k}^{\pm1},\bar b_{n,p^k}^{\pm1}\}\bigr),
\]
where $\bar a_{n,p^k},\bar b_{n,p^k}$ denote the reductions of $a_n,b_n$ modulo $p^k$.

\begin{remark}\label{rem:bounded-girth}
For fixed $n\ge 3$ and prime $p$, the family $(X_{n,p^k})_{k\ge 1}$ is an expander
family by the property~\textup{($\tau$)} consequence of property~\textup{(T)} for
$\SL_n(\ZZ)$ with respect to the principal congruence subgroups; see
\cite[Chapter~1]{BekkaValette} and \cite[Chapter~4]{LubotzkyExp}. Its girth is uniformly
bounded because $\bar a_{n,p^k}^{\,2n}=1$ by Corollary~\ref{cor:modm-orders}, so each
graph contains a nontrivial closed walk of length at most $2n$, and hence a cycle of
length at most $2n$ after passing to the corresponding reduced cycle. Thus this fixed-generator congruence family is very
different from logarithmic-girth constructions based on fixed matrices generating a free
subgroup, such as those of Arzhantseva--Biswas \cite{ABgirth}.
\end{remark}

\section{Related generating families and concluding remarks}\label{sec:extensions}

\subsection{Odd rank and Trott's pair}

When $n$ is odd, one has $\eps_n=1$, so the matrix $a_n$ in \eqref{eq:anbn} is exactly
the cyclic permutation matrix used by Trott \cite{Trott}. Thus in odd rank the pair
$(a_n,b_n)$ is precisely Trott's pair. In that sense, the present construction may be
viewed as a presentation-theoretic extraction from Trott's generating theorem, extended
uniformly to even rank by the sign-corrected wrap-around entry.

Viewed against the two classical uniform generating constructions, the family $(a_n,b_n)$
therefore sits in between them: in odd rank it is exactly Trott's pair, while across all
ranks $n\ge 3$ it avoids the rank-$4$ exception in Gow--Tamburini's Jordan/transposed-
Jordan family.

\subsection{Concluding remarks}

The construction gives a uniform Steinberg-based route from the transvection presentation
to explicit two-generator presentations for the monomial/transvection pair $(a_n,b_n)$.
The balanced recursion provides polynomial control on the resulting words and relators.
The additional variants and quotient presentations are formal consequences of the same
explicit transvection recovery. We do not address finite-index subgroups. In a different direction, Meiri proved that
every finite-index subgroup of \(\SL_n(\ZZ)\) contains a further finite-index subgroup
generated by two elements \cite{Meiri}.

\section*{Acknowledgements}
The formal Tietze-elimination step of Proposition~\ref{prop:tietze-template} was
first stated in the MathSciNet review of the paper of Conder--Liversidge--Vsemirnov by Jack Button, with attribution to Philip Hall. I thank him for drawing my attention to this.

\end{document}